\documentclass{article}

\usepackage{amsmath, amssymb, amsthm}

\newtheorem{Theorem}{Theorem}[section]
\newtheorem{Corollary}[Theorem]{Corollary}
\newtheorem{prop}[Theorem]{Proposition}
\newtheorem{lem}[Theorem]{Lemma}
\theoremstyle{definition}
\newtheorem{remark}[Theorem]{Remark}
\newtheorem{Definition}[Theorem]{Definition}

\renewcommand{\leq}{\leqslant}

\renewcommand{\bar}{\overline}

\DeclareMathOperator{\ih}{ih}

\author{Jonathan A. Kiehlmann 
\\ Dept.\ of Mathematics
\\ Imperial College London
\\ South Kensington Campus
\\ London SW7 2AZ 
\\ jonathan.kiehlmann@gmail.com}
\title{Classifications of Countably Based Abelian Profinite Groups}
\begin{document}

\maketitle

\begin{abstract}
 In the first half of this paper, we outline the construction of a new class of abelian pro-$p$ groups, which covers all countably based pro-$p$ groups.  
In the second half, we study them, and classify them up to topological isomorphism and abstract isomorphism.  
Topologically, any infinitely generated abelian pro-$p$ group decomposes as a product of infinitely many non-trival groups.  
It emerges that they are all abstractly isomorphic to Cartesian products of finite groups and $p$-adic integers.  
We have thus constructed uncountably many pairwise topologically non-isomorphic profinite groups abstractly isomorphic to a Cartesian product of cyclic groups.
\end{abstract}

In \cite{NS1}, Nikolov and Segal show that finitely generated profinite groups are strongly complete: that is, that their topology is entirely determined by their algebraic structure.  
Equivalently, any homomorphism from a finitely generated profinite group $G$ to a profinite group is continuous.  
In \cite{NS2}, Nikolov and Segal examine the images of possible non-continous homomorphisms from profinite groups.  
We investigate a question suggested by Segal: do there exist profinite groups which are abstractly, but not continuously, isomorphic?  
Through investigating the abelian case, we produce classifications of countably based abelian profinite groups, up to continuous isomorphism and up to abstract isomorphism.   
In particular, any countably based pro-$p$ group which contains torsion of unbounded exponent is isomorphic to the closure of its torsion elements, which is isomorphic to a Cartesian product of cyclic $p$-groups.  

This classifies all topological abelian group structures which can be put on the Cantor set.  
On the other hand, if $\alpha$ is an uncountable cardinal such that $2^{\alpha}=2^{\aleph_{0}}$,it is not hard to see that  $\prod_{\alpha}C_{p}$ and $\prod_{\aleph_{0}}C_{p}$ are an example of abstractly isomorphic, but topologically non-isomorphic profinite groups.  

Recall that a topological group is countably based if a point has a countable neighbourhood basis.  
As the open normal subgroups of a profinite group provide a neighbourhood base for the identity, a profinite group is countably based if and only if it has countably many open normal subgroups.  
This is equivalent to having countably many finite continuous images.  
In the abelian case, this is equivalent to having a countable Pontryagin dual.  
So profinite countably based abelian groups are dual to countable discrete torsion groups.  
However these have been classified.  

We apply Pontryagin duality to the traditional classification of countable abelian torsion groups, by Pr\"{u}fer, Ulm and Zippin, as outlined in \cite{Kaplansky} or \cite[chapter VI]{Fuchs}.  
The definition of the torsion sequence of a profinite group is new, but very natural.

In Section \ref{Torsion}, we introduce the definition of the torsion sequence of a profinite group, which is key to the structure of profinite abelian groups.  
In Section \ref{HowToBuild}, we describe the classification of countably-based abelian pro-$p$ groups and make some observations about their structure.  
In Section \ref{SI}, we determine which of these topological groups are abstractly isomorphic.  
Finally, in Section \ref{Universality} we show that any abelian pro-$p$ group with unbounded torsion contains closed subgroups continuously isomorphic to each abelian pro-$p$ group.  

The construction of Section \ref{HowToBuild} generalises to an uncountable case with a little work.  
That is, we can construct groups with torsion sequences of uncountable torsion type or with uncountably-based Cartesian factors.  
These are precisely the duals of the totally projective $p$-groups and the same results on abstract isomorphisms hold.  
This will be expanded on in \cite{JK}.

\begin{section}*{Notation}

Throughout, for $G_i, i\in I$ a family of groups, we write
$$\prod_{i\in I}G_{i}$$
for the (unrestricted Cartesian) product of the $G_{i}$, i.e., the set $\{(g_{i})_{i\in I} : g_{i}\in G_{i}\}$ with termwise addition (or multiplication).  
We write 
$$\bigoplus_{i\in I} G_{i}$$
for the direct sum (or restricted product) of the $G_{i}$, that is, the subgroup of the product consisting of strings $(g_{i})_{i\in I}$ with $g_i$ trivial for all but finitely many $i$.  
For $G$ a group and $\alpha$ a cardinal,  $(G)^{\alpha}$ denotes the product of $\alpha$ copies of $G$.  
For $n\in\mathbb{N}$, we write
$$G^n:=\{x^n\mid x\in G\}$$
for the $n$th powers in $G$, 
$$G[n]:=\{x\in G\mid x^n=1\}$$ 
for the elements of order dividing $n$ in $G$ and
$$t(G)=\bigcup_{n\in\mathbb{N}}G[n] $$
for the set of torsion elements in $G$.  
We shall say a that a group has unbounded torsion if it has torsion elements of unbounded exponent, and we shall say a sequence $(a_{i})_{i\in\mathbb{N}}$ with entries in $\mathbb{N}$ is unbounded if there is no upper bound for $\{a_{i}: i\in\mathbb{N}\}$.  

If $G$ is a profinite or discrete torsion group, $G^{*}$ denotes the Pontryagin dual of $G$.

Finally,
$$H\leq_{C} G$$
indicates that $H$ is a closed subgroup of $G$.  

\end{section}

\begin{section}{Torsion Sequences}\label{Torsion}

We classify abelian profinite groups by considering their Pontraygin duals.  
A profinite abelian group is the product of its pro-$p$ subgroups, and so we need only consider pro-$p$ groups.  

Recall that an element $x$ of a torsion group $T$ is said to be of infinite height if for every $n\in\mathbb{N}$ there is some $y$ in $T$ with $yn=x$.  

\begin{Theorem}\label{ClToInfHei}
Let $G$ be an abelian profinite group or a discrete abelian torsion group.  
Then,
$$(\bar{t(G)})^{*} \cong G^{*}/\bar{\ih(G^{*})}$$
as topological groups, where $\ih(G^{*})$ denotes the set of elements of infinite height in $G^{*}$.
\end{Theorem}
This is a strengthening of (\cite{RZ}, Theorem 2.9.12).  
\begin{proof}
By definition, Ann$_{G^{*}}(\bar{t(G)})$ is the least closed set containing 
$$\bigcap_{n\in\mathbb{N}}\text{Ann}_{G^{*}}(G[n])$$
which, by (\cite{RZ}, Lemma 2.9.11), is equal to
$$\bigcap_{n\in\mathbb{N}}{n}(G^{*})$$
which is of course the collection of all elements of infinite height in $G^{*}$.
\end{proof}
From this we can see that torsion-free pro-$p$ groups are dual to divisible $p$-groups (those where all elements are of infinite height).  
Divisible abelian groups are classified; a divisible abelian group is a direct sum of copies of the additive group of $\mathbb{Q}$ and of quasicyclic groups.  
From this, we can immediately derive the following.  

\begin{Corollary}
Let $\phi:G\to H$ be a continuous homomorphism from an abelian pro-$p$ group to a torsion-free group.  
Then there is some closed subgroup $K$ of $G$ isomorphic to $\prod_{I}\mathbb{Z}_{p}$ (isomorphic to the image of $\phi$), for some index set $I$, with 
$$G=K\times\ker\phi .$$
\end{Corollary}
This follows immediately from Pontryagin duality and the fact that a divisible subgroup of an abelian is a direct summand.  
Hence all torsion-free quotients of abelian pro-$p$ groups split, and all torsion-free abelian pro-$p$ groups are free.  

This is stronger than (\cite{Serre}, I.1, Exercise 1).  
This corollary says that any torsion-free quotient of an abelian pro-$p$ group has a complement, and so splits as a direct factor.  
All profinite groups are reduced in the traditional sense of having trivial maximal divisible subgroup; a divisible subgroup of any group must be contained in the finite residual, which is trivial, by residual finiteness.  
\begin{Definition}
We shall say a profinite group is \emph{dual-reduced} if it has no non-trivial continuous torsion-free quotients: Theorem \ref{ClToInfHei}, shows this is equivalent to having a reduced dual.
\end{Definition}

\begin{Theorem}\label{Peeling}
Let $G$ be a pro-$p$ abelian group.  
Then 
$$G=G_{0}\times (\mathbb{Z}_{p})^{X}$$
for some cardinal $X$, where $G_{0}$ is a maximal dual-reduced subgroup.
\end{Theorem}
As $\mathbb{Z}_{p}^{*}$ is isomorphic to the quasicyclic group $C_{p^{\infty}}$, this follows from considering the dual of a discrete $p$-group, which must decompose as a sum of a reduced and (maximal) divisible part.

We introduce a topologically characteristic series of subgroups of a profinite group. 
\begin{Definition}
Let $G$ be an arbitrary abelian pro-$p$ group.

Set $T_{0}(G)$ to be the trivial subgroup.  
Now, we recursively define
$$T_{\alpha+1}(G)/T_{\alpha}(G)=\bar{t(G/T_{\alpha}(G)}$$
for any ordinal $\alpha$ and
$$T_{\delta}(G)=\bar{\langle  T_{\alpha}(G): \alpha<\delta\rangle}$$
for $\delta$ any limit ordinal.  

We call this chain the \emph{torsion series} of $G$.  
As this chain is increasing there will be some ordinal $\tau$ such that $T_{\tau}(G)=T_{\tau+1}(G)$.  
We call the least such $\tau$ the \emph{torsion type} of $G$.

We set $G_{T_{\alpha}}$ to be $T_{\alpha+1}(G)/T_{\alpha}(G)$.  
Now, we call the well-ordered transfinite sequence of quotients  
$$ G_{T_{0}}, G_{T_{1}}, \ldots, G_{T_{\alpha}}, \ldots$$
the \emph{torsion sequence} of $G$.
\end{Definition}

\begin{remark}
This sequence, by Theorem \ref{ClToInfHei}, is dual to the Ulm sequence of a discrete torsion group (as defined in \cite{Fuchs} \S35).  
That is, the $\alpha$-th term of the Ulm sequence of $G^{*}$ is isomorphic to the dual of $G_{T_{\alpha}}$.   
Ulm's Theorem (\cite{Fuchs}, Theorem 37.1) says that the isomorphism type of a countable group is determined by its Ulm sequence.  
\end{remark}
\begin{remark}
The torsion type of a group is not always the supremum of the torsion types of its subgroups.  
Indeed, Section \ref{Universality} shows that any countably based dual-reduced pro-$p$ group is isomorphic to a closed subgroup of any countably based dual-reduced pro-$p$ group which contains torsion of unbounded order.  
On the other hand, all continuous quotients of countably based dual-reduced pro-$p$ groups will have torsion types bounded by that of the original group.  
\end{remark}

The torsion sequence completely determines the topological-group structure of the dual-reduced part of a countably-based abelian pro-$p$ group.  

\begin{Theorem}\label{proUlm}
Let $G, H$ be countably based abelian dual-reduced pro-$p$ groups with the same torsion sequence.  
Then $G$ and $H$ are isomorphic (as topological groups).
\end{Theorem}
By ``the same torsion sequence" we mean that $G$ and $H$ are of the same torsion type and that $G_{T_{\alpha}}\cong H_{T_{\alpha}}$ for each $\alpha$.

\begin{proof} 
The dual groups $G^{*}, H^{*}$ are countable abelian $p$-groups.   
Repeated applications of Theorem \ref{ClToInfHei}, shows that the dual of the torsion sequence of a pro-$p$ group is isomorphic to Ulm sequence of its dual.  
Hence, by Ulm's Theorem, (\cite{Fuchs}, Theorem 37.1),  $G^{*} \cong H^{*}$.  
Now, by Pontryagin duality $G$ and $H$ are topologically isomorphic.  
\end{proof}

An obvious question is ``Can we construct pro-$p$ groups with arbitrary torsion sequences?".  
We can answer this positively.  

\end{section}

\begin{section}{How to Build Abelian Pro-$p$ Groups}\label{HowToBuild}
In the following section we describe how to construct an arbitrary totally injective abelian pro-$p$ %(and hence, profinite) group.
group.  

\begin{Definition}
For $X,Y$ abelian profinite groups $x\in X, y\in Y$, we write \emph{the profinite presentation}
$$\langle X,Y : x=y\rangle$$
to denote the quotient $X\times Y/\bar{\langle x-y\rangle}$.  
\end{Definition}

We will only use this in a case where $Y=\bar{\langle y\rangle}\cong\mathbb{Z}_{p}$ , the additive group of the $p$-adic integers, and with a relation of the form $p^{a}y=z$, for some infinite order element of $X\setminus(t(X) + {p}X)$.  
It is elementary to see that this produces a profinite abelian extension of $X$ by $C_{p^{a}}$.
  
\begin{Definition}
For convenience, we shall say an abelian pro-$p$ group is \emph{Cartesian} if it is a Cartesian product of a collection of cyclic $p$-groups. 
\end{Definition}

It is helpful to consider only those profinite groups with closure of torsion elements a Cartesian product of cyclic $p$-groups.  
It transpires that this class includes all countably based pro-$p$ groups.  

\begin{prop}
Let $G$ be a countably based pro-$p$ group.  
Then $\bar{t(G)}$, and hence each term of the torsion sequence of $G$ is Cartesian.  
\end{prop}

\begin{proof}
By induction, it is sufficient only to consider $\bar{t(G)}$.  

If $G$ is a countably based pro-$p$ group, then $G^{*}$ is a countable discrete abelian $p$-group.  
Now, $(\bar{t(G)})^{*}$ will be isomorphic to $G^{*}/\ih(G^{*})$, which is a countable abelian $p$-group without non-trivial elements of infinite height.  
By a well known result of Pr\"{u}fer, (\cite{Fuchs}, Theorem 11.3) this must be isomorphic to a direct sum of cyclic $p$-groups.  
The dual of a direct sum of finite abelian $p$-groups is isomorphic to the Cartesian product of these groups (\cite{RZ}, Theorem 2.9.4). 
Hence $\bar{t(G)}$ will be isomorphic to a Cartesian product of cyclic $p$-groups, that is, Cartesian.  
\end{proof}

\begin{Theorem}\label{Construction}
Let $\tau$ be a countable ordinal.  
For $\alpha<\tau$, let $N_{\alpha}$ a Cartesian abelian pro-$p$ group, such that every non-final term of the sequence $(N_{\alpha})$ is of unbounded exponent.  
Then there exists a dual-reduced abelian pro-$p$ group $G$ with torsion sequence $(N_{\alpha})$.
\end{Theorem}

Our proof comes largely from applying Pontryagin duality to the proof of Zippin's Theorem given in (\cite{Fuchs}, \S36).  
We give this dual construction (which does not seem to have appeared previously in the literature) as analysis of this construction is useful for proving Theorem \ref{SAI}.  

\begin{proof}
We prove the first statement and proceed by induction on $\tau$.  

The base case is $\tau=2$, where we have one group which is the closure of its torsion elements, and  $G=N_{0} is$ a Cartesian group.  

We now split into five cases.
\newline
Case I: $\tau-2$ exists, and $N_{\tau-1}$ is a cyclic group, of order $p^{r}$.

By induction, there exists a group $H$ with torsion sequence $(N_{0}, \ldots, N_{\tau-2})$.  
Now, $H/T_{\tau-2}(H)\cong \prod_{i\in\mathbb{N}} X_{i}$, where each $X_{i}$ is a cyclic group.  
So we have canonical epimorphisms $\theta_{i}:H\to X_{i}$, for $i\in\mathbb{N}$.  
Choose $\delta\in H$ which is sent by each $\theta_{i}$ to an element in $X_{i}$ which generates that group.  
Then, take the abelian group with profinite presentation (as defined above)
$$G=  \langle H, X=\bar{\langle x\rangle} : p^{r}x=\delta\rangle$$
where $X$ is an infinite procyclic pro-$p$ group topologically generated by $x$.  

We claim that $T_{\alpha}(G)=T_{\alpha}(H)$ for each $\alpha < \tau -1$.  
Suppose otherwise.  
Take $\beta\leq \tau-2$ to be minimal such that $T_{\beta}(G)\neq T_{\beta}(H)$.  
If $\beta$ is a limit ordinal
$$T_{\beta}(H) = \bar{\bigcup_{\mu<\beta} T_{\mu}(H)} = \bar{\bigcup_{\mu<\beta} T_{\mu}(G)} = T_{\beta}(G)$$
and we have a contradication.  

So $\beta$ is a successor ordinal: write $\beta=\gamma +1$.  

Note that $T_{\gamma+1}(G)$ is not contained in $H$.  
Otherwise, 
$$T_{\gamma+1}(G)/T_{\gamma}(G) = T_{\gamma+1}(G)/T_{\gamma}(H)$$ 
is a closed subgroup of $H/T_{\gamma}(H)$, and so $T_{\gamma+1}(G)\subseteq T_{\gamma+1}(H)$.  
Then 
$$T_{\gamma+1}(G)=T_{\gamma+1}(H),$$ 
contradicting  our assumption on $\beta$.  

Take some $a\in T_{\gamma+1}(G)\setminus H$ which is of finite order modulo $T_{\gamma}(G)=T_{\gamma}(H)$.  
As $\gamma=\beta-1 \leq \tau-2$, $a+ T_{\tau-2}(H)$ is of finite order.  
So, we have $a\in n\delta+H$ for some $0 <n\leq p^{r}-1$, of form $n=p^{k}s$, for some $s$ coprime to $p$.  
Then ${p^{r-k}}a\in H$, and so ${p^{r-k}}a\equiv \delta^{s}$ mod $pH$.  
But we chose $\delta$ so that $\delta + T_{\tau-2}(H)$ must be torsion-free in $H/T_{\tau-2}(H)$, so $\delta^{s}+pH$ cannot contain elements of finite order in $H/T_{\tau-2}(H)$.  
This is a contradiction.  
Hence $T_{\tau-1}(G)=H$ and $T_{\alpha}(G)=T_{\alpha}(H)$, for each $\alpha\leq \tau-1$.  

Thus, as $G/H$ is cyclic of order $p^{r}$, it follows that $G$ has the required torsion sequence.  
\newline
Case II: $\tau-2$ exists, and $N_{\tau-1}$ is not cyclic.  

We can find cyclic groups, $\{N_{\tau-1,n}\}_{n\in\mathbb{N}}$ such that $N_{\tau-1}=\prod_{n\in\mathbb{N}}N_{\tau-1, n}$.  
For each $\varepsilon < \tau-1$, we can find some decomposition $N_{\varepsilon}=\prod_{n\in\mathbb{N}}N_{\varepsilon, n}$ with each $N_{\varepsilon, n}$ Cartesian of unbounded exponent.  
For example, if $N_{\varepsilon}=\prod_{i\in\mathbb{N}}C_{p^{i}}$, we can take any partition of $\mathbb{N}$ into infinitely many pairwise disjoint subsets $(K_{t})_{t\in T}$, with $\bigcup_{t\in T} K_{t}=\mathbb{N}$ and write $N_{\epsilon}=\prod_{t\in T}\prod_{i\in K_{t}} C_{p^{i}}$.     

Now, by Case I, we can construct groups $H_{n}$ with Ulm sequence $N_{0,n}, \ldots, N_{\tau-1, n}$, for each $n$.  
Now 
$$G=\prod_{n\in\mathbb{N}}H_{n}$$
 has the required torsion sequence, as the product $\prod _{n\in\mathbb{N}}T_{\alpha}(H_{n})$ will be isomorphic to $T_{\alpha}(\prod_{n\in\mathbb{N}}H_{n})$.  
\newline
Case III: $\tau$ is a limit ordinal.  

Take a set $\{\sigma_{n}: n\in\mathbb{N}\}$ of ordinals less than and with supremum $\sigma$.

We can use methods as in Case II above to find Cartesian groups $N_{\varepsilon,n}$ such that $\prod_{n}N_{\varepsilon,n}=N_\varepsilon$ with each $N_{\varepsilon,n}$ is of unbounded exponent if $\varepsilon<\sigma_{n}$ and trivial if and only if $\sigma_{n}<\varepsilon$.  

By inductive hypothesis we can construct a group $H_{n}$ with torsion sequence $(N_{\varepsilon,n})$, for each $n$.  
The group 
$$G=\prod_{n\in\mathbb{N}}H_{n}$$
has the required torsion sequence, as the torsion sequence of a product is the product of torsion sequences, as in Case II.
\newline
Case IV: $\tau-1$ exists and is a limit ordinal, $\sigma$, and $N_{\sigma}$ is cyclic of order $p^{r}$.  

Take a set $\{\sigma_{n}: n\in\mathbb{N}\}$ of ordinals less than and with supremum $\sigma$.  

Construct, as in Case III, new Cartesian groups $N_{\varepsilon,n}$ with $\prod_{n}N_{\varepsilon,n}=N_\varepsilon$ such that $N_{\varepsilon,n}$ is of unbounded exponent if $\varepsilon\leq\sigma_{n}$ and trivial if $\sigma_{n}<\varepsilon$.  
Furthermore, without loss of generality, we can choose each $N_{\sigma_{n},n}$ to be cyclic.  
By the inductive hypothesis we can construct a group $H_{n}$ with torsion sequence $(N_{\varepsilon,n})$, for each $n$.  
For each $n$, we can choose $\gamma_{n}\in H_{n}$ such that $\gamma_{n}+T_{\sigma_{n}-1}( H_{n}) $ is a generator for $H_{n}/T_{\varepsilon_{n}-1}( H_{n})$.  
Choose $\gamma\in H=\prod_{n}H_{n}$ such that, for each $n$, the canonical epimorphism $\phi_{n}:H\to H_{n}$ sends $\gamma$ to $\gamma_{n}$.  
Note that $H$ has torsion type $\tau-1$ and torsion sequence $(N_{\varepsilon})_{\varepsilon<\tau}$.  

In the same way as Case I, we construct an extension of a Cartesian product of groups by $C_{p^{r}}$.  
Consider
$$G=\langle\prod_{n}H_{n}, X=\bar{\langle x\rangle} : p^{r}x=\gamma\rangle$$
where $X$ is an infinite procyclic pro-$p$ group topologically generated by $x$.  
Now $G/H\cong C_{p^{r}}$ and so to show $G$ has the required torsion sequence, it remains only to show $T_{\alpha}(G)=T_{\alpha}(H)$ for each $\alpha \leq \tau$.  

Suppose otherwise: we can pick least ordinal $\beta$ with $T_{\beta}(G)\neq T_{\beta}(H)$.  
This must be a successor ordinal, $\beta$, by the definition of the torsion sequence.  
Pick some $n$ with $\sigma_{n}>\beta$.  
Now, for $a \in T_{\beta}(G)\setminus H$, $a$ must be a torsion element modulo $T_{\sigma_{n}}(H)$.  
We have $a\in n\gamma+ H$, for some $0 <n\leq p^{r}-1$, of form $n=p^{k}s$, for some $s$ coprime to $p$.  
Then ${p^{r-k}}a\in H$, and so ${p^{r-k}}a\equiv \gamma^{s}$ mod $pH$.  
But we chose $\gamma$ so that $\gamma + T_{\sigma_{n}}(H)$ must be torsion-free in $H/T_{\sigma_{n}}(H)$, so $\gamma^{s}+pH$ cannot contain elements of finite order in $H/T_{\sigma_{n}}(H)$.  
This is a contradiction.  
Hence $T_{\alpha}(G)=T_{\alpha}(H)$ for each $\alpha\leq \tau-1$ and $T_{\tau-1}(G)=H$.  
\newline
Case V: $\tau-1$ exists and is a limit ordinal, $\sigma$, and $N_{\sigma}$ is not cyclic.  

This follows from Case IV in exactly the same way as Case II follows from Case I.  
As in Case II, we find cyclic groups, $\{N_{\sigma,n}\}_{n\in\mathbb{N}}$ such that $N_{\sigma}=\prod_{n\in\mathbb{N}}N_{\sigma, n}$.  
For each $\varepsilon < \sigma$, we can find some decomposition $N_{\varepsilon}=\prod_{n\in\mathbb{N}}N_{\varepsilon, n}$ with each $N_{\varepsilon, n}$ Cartesian of unbounded exponent.  

By the inductive hypothesis we can construct a group $H_{n}$ with torsion sequence $(N_{\alpha,n})_{\alpha \leq \sigma}$, for each $n$.  
But then 
$$G=\prod_{n\in\mathbb{N}}H_{n}$$
has the required torsion sequence, as before.
\end{proof}

In the proof, we use the axiom of choice extensively and have many choices of ways to decompose Cartesian groups of unbounded exponent as products of infinitely many Cartesian groups of unbounded exponent.  
The choices made at each stage do not matter, as Theorem \ref{proUlm} shows.  

By Theorem \ref{Peeling}, Theorem \ref{Construction} and Theorem \ref{proUlm} we have thus classified all countably based abelian dual-reduced pro-$p$ groups.  
Hence we can construct any countably based abelian profinite group as a product of its Sylow pro-$p$ subgroups.  

This completely classifies countably based profinite groups up to continuous isomorphism.

\begin{prop}\label{Prod}
Let $G$ be an infinite countably based dual-reduced abelian pro-$p$ group of torsion type $\tau$.  
Then 
\begin{enumerate}
\item We have $G=\prod_{i\in I} K_{i}$ for infinite index set $I$ and non-trivial closed subgroups $K_{i}$.  
\item Furthermore, we can chose $K_{i}$ such that 
\begin{enumerate} 
\item If $G_{T_{\tau-1}}$ is finite, each $K_{i}$ is of torsion type at least $\tau-1$.  
\item If $G_{T_{\tau-1}}$ is infinite or $\tau$ is a limit ordinal, each $K_{i}$ is of torsion type $\tau$.  
\end{enumerate}
\item We can choose the $K_{i}$ such that each $K_{i}$ of torsion type some successor ordinal $\alpha_{i}$ has $(K_{i})_{T_{\alpha_{i}}}$ cyclic.
\end{enumerate}
\end{prop}
This is stronger than the result dual to (\cite{Fuchs}, Proposition 77.5), and, aside from the use of Ulm's theorem in our classification, proves it independently.  
\begin{proof}
We prove the first two parts together.
As $G$ is countably based, we can see that each $G_{T_{\alpha}}$ must be Cartesian.  
Now, we can write each $G_{T_{\alpha}}=\prod_{n\in\mathbb{N}} H_{\alpha,n}$ subject to the following conditions: for each $\alpha,n$
\begin{enumerate}
\item $H_{\alpha,n}$ is Cartesian,
\item $H_{\alpha,n}$ is of unbounded exponent whenever $G_{T_{\alpha}}$ is of unbounded exponent,
\item $H_{\alpha,n}$ is finite if and only if $G_{T_{\alpha}}$ is finite.  
\end{enumerate}
Note that $H_{\alpha+1,n}$ is thus non-trivial only if $H_{\alpha,n}$ is of unbounded exponent.  
By Theorem \ref{Construction}, we can find groups $H_{n}$ with torsion sequence $(H_{\alpha,n})_{\alpha}$.  
Now, by Theorem \ref{proUlm}, we have $G\cong\prod_{n\in\mathbb{N}}H_{n}$.   

We claim that the $H_{n}$ are all non-trivial, and satisfy the conditions on the $K_{i}$ given in the second part of the statement.  
We proceed by induction on $\tau$.  
As $G$ is infinite, $T_{G_{0}}=\bar{t(G)}$ must also be infinite.  
As an infinite Cartesian group, it is by definition a product of infinitely many non-trivial groups, and so in the base case, $\tau=2$, the proposition holds.  

If $\tau>1$, $H_{0,n}$ is of unbounded exponent and so non-trivial, for each $n$, and so each $H_{n}$ must also be of torsion type at least $1$.  
If $\tau$ is a limit ordinal, then, for each $\alpha<\tau$, $G_{T_{\alpha}}$ must be of unbounded exponent, and so each $H_{\alpha,n}$ must be of unbounded exponent.  
Hence each $H_{n}$ must be of torsion type $\tau$.  
If $\tau$ is a successor ordinal and $G_{T_{\tau-1}}$ is finite then all but finitely many $H_{n}$ will be of torsion type $\tau-1$: the remainder will be of type $\tau$.  
Otherwise, if $G_{T_{\tau-1}}$ is infinitely generated, each $H_{\tau-1,n}$ must also be infinitely generated, and so each $H_{n}$ must be of torsion type $\tau$.  
Thus $G\cong \prod_{n\in\mathbb{N}}H_{n}$, and so we can find closed subgroups $K_{i}$ satisfying the first two parts of the statement.

We can decompose the $G_{T_{\alpha}}$ in different ways, and thus get decompositions with different properties.  
In particular, we can always pick $H_{\alpha,n}$ such that each $H_{n}$ of torsion type $\tau_{n}=\varepsilon_{n}+1$ has cyclic $(H_{n})_{T_{\varepsilon_{n}}}$.  
By the same argument as above, this proves the third part, completing the proof.
\end{proof} 

\begin{Corollary}
If $G$ is an infinitely generated, countably based abelian profinite group, $G$ can be decomposed as a product of infinitely many non-trivial groups. 
\end{Corollary}
Note the converse of this is not true: the group $\prod_{p\text{ prime}}C_{p}$ is topologically $1$-generated.  

\begin{proof}
An abelian profinite group $G$ is equal to $\prod_{p\text{ prime}} G_{p}$, where $G_{p}$ is a $p$-Sylow subgroup of $G$.  
If $G$ is infinitely generated, either infinitely many of these are non-trivial, and we are done, or there is some infinitely generated $G_{p}$.  
Consider $G_{p}$, an infinitely generated pro-$p$ group.  
By Theorem \ref{Peeling}, $G_{p}=H\times F$, where $H$ is dual-reduced and $F$ is torsion-free, and so one of these must be infinitely generated.  
As $F$, must be a free abelian pro-$p$ group, isomorphic to $\prod_{i\in I}\mathbb{Z}_{p}$, so a product of infinite procyclic groups.  
If $F$ is infinitely generated, it is therefore the product of infinitely many infinite procyclic groups and our result holds.  
If $H$ is infinitely generated it must be countably based, as $G$ is countably based.  
And now Proposition \ref{Prod}\ gives our result.  
\end{proof}

We now look at which of these groups are isomorphic as abstract groups.

\end{section}

\begin{section}{Strange Isomorphisms}\label{SI}

\begin{Theorem}{Strange Abelian Isomorphisms}\label{SAI}

Let $G$ be a countably based abelian pro-$p$ group.  
If ${t(G)}$, the torsion subgroup of $G$, is of finite exponent then it is a closed subgroup of $G$ and $G$ is of form
$$\prod_{i=1}^{e}(C_{p^{i}})^{\alpha_{i}} \times \mathbb{Z}_{p}^{X}$$
for some $e\in\mathbb{N}$, $(\alpha_{i})$ a sequence in $\mathbb{N}\cup\{0,\aleph_{0}\}$ and $X$ some cardinal.  
Otherwise, $G$ is isomorphic to $\bar{t(G)}$ as an abstract group and $\bar{t(G)}$ is of form
$$\prod _{i\in\mathbb{N}} (C_{p^i})^{\alpha_{i}},$$ for $(\alpha_{i})$ a sequence in $\mathbb{N}\cup\{0,\aleph_{0}\}$ not tending to $0$.  

\end{Theorem}

In fact, it is not hard to see that $\alpha_{i} =\dim_{\mathbb{F}_{p}} (G[p^{i}]/(pG\cap G[p^{i}]))$.  
(This function is the pro-$p$ equivalent of the Ulm invariant function.)  
Specifically, the $\alpha_{i}$ are invariants of the torsion part of $G$, and so do not depend on the topological structure of $G$.

\begin{Corollary}
Let $G, H$ be countably based abelian pro-$p$ groups with unbounded torsion.  
Then $G$ and $H$ are abstractly isomorphic if and only if $t(G)$ is isomorphic to $t(H)$.
\end{Corollary}

\begin{proof}
As $G$ has unbounded torsion, by Theorem \ref{SAI}, it is abstractly isomorphic to $\prod C_{p^{\alpha_{i}}}$, for some sequence $(\alpha_{i})$.  
The values of the $\alpha_{i}$ depend only on the abstract isomorphism class of $t(G)$, the abstract torsion part of $G$.  
Hence the abstract isomorphism class of $G$ is totally determined by the abstract isomorphism class of $t(G)$, which also determines the abstract isomorphism class of $\bar t(G)$.  
\end{proof}

If two abelian countably based pro-$p$ groups have isomorphic torsion subgroups, they are abstractly isomorphic if they have unbounded torsion.  
Otherwise, they are abstractly isomorphic if and only if they are isomorphic as topological groups.

To prove Theorem \ref{SAI} , we first prove the following crucial lemma:

\begin{lem}\label{NikHom}

Let $(\alpha_{n})_{n\in\mathbb{N}}$ be a sequence in $\mathbb{N}\cup\{0,\aleph_{0}\}$ not tending to 0.  

Then 
$$ \prod _{i\in\mathbb{N}} (\mathbb{Z}/{p^i}\mathbb{Z})^{\alpha_{i}}\cong \prod _{i\in\mathbb{N}} (\mathbb{Z}/{p^i}\mathbb{Z})^{\alpha_{i}}\times\prod_{I} \mathbb{Z}_p $$
as abstract groups, for arbitrary countable index set $I$.

\end{lem}

We do this by using a non-principal ultrafilter $\mathcal{U}$ on $\mathbb{N}$ to define a series of homomorphisms in a similiar way to an ultralimit. 

\begin{Definition} 
An \emph{ultrafilter} $\mathcal{U}$ on a set $X$ is a subset of the power set of  $X$ such that the function defined by $m(A)=1$ if $A\in\mathcal{U}$, $m(A)=0$ otherwise, is a finitely additive measure.  
It is said to be \emph{principal} if it has some least element, $\{x\}\subseteq X$ and \emph{non-principal} otherwise.  
\end{Definition}

The existence of non-principal ultrafilters on infinite sets is equivalent to the Boolean Prime Ideal Theorem.  

\begin{proof}

Without loss of generality, we can assume each $\alpha_{i} \in \{0,1\}$, by considering a direct factor of a partial product.  
Define a set of homomorphisms $\phi^{i}_{j}: (\mathbb{Z}/{p^i}\mathbb{Z})^{\alpha_i}\to \mathbb{Z}/{p^j}\mathbb{Z}$ by 
$$\phi_{j}^{i}(x) \equiv x \pmod {p^j}\textup{ for }j\leq i,\textup{ and }x\phi^{i}_{j}=0\textup{ for }j>i.$$  

For $(x_1, x_2, \ldots)\in  \prod_{i\in \mathbb{N}} (\mathbb{Z}/{p^i}\mathbb{Z})^{\alpha_i}$, set
$$I^{n}_{j}((x_1, x_2, \ldots))= \{m\in \mathbb{N}\mid x_m\phi^{i}_j = n\}$$
giving, for each $j\in\mathbb{N}$, a function from our Cartesian group to the power set of $\mathbb{N}$.  

We define a new map by setting $(x_1, x_2, \ldots) \psi^{(\mathcal{U})}_{j}$ to be the unique $n\in\mathbb{N}$ such that $I^{n}_{j}((x_1, x_2, \ldots))$ is in $\mathcal{U}$.

Now, for $x,y\in \prod_{i\in \mathbb{N}} (\mathbb{Z}/{p^i}\mathbb{Z})^{\alpha_i}$, we have
$$I_{j}^{x\psi^{\mathcal{U}}_{j}}(x)\cap I_{j}^{y\psi^{\mathcal{U}}_{j}}(y) \subseteq
I_{j}^{x\psi^{\mathcal{U}}_{j}+y\psi^{\mathcal{U}}_{j}}(x+y).$$
The left hand side is the intersection of two elements of $\mathcal{U}$, and so is in $\mathcal{U}$, as it is an ultrafilter.  
Thus the right hand side contains an element of $\mathcal{U}$, and so is also in $\mathcal{U}$, and so we have
$$(x+y)\psi_{j}^{\mathcal{U}}=x\psi_{j}^{\mathcal{U}}+y\psi_{j}^{\mathcal{U}}.$$

So $\psi^{(\mathcal{U})}_{j}: \prod_{i\in \mathbb{N}} (\mathbb{Z}/{p^i}\mathbb{Z})^{\alpha_i} \to\mathbb{Z}/{p^j}\mathbb{Z} $ is a homomorphism.  
Now, the $\mathbb{Z}/p^{j}\mathbb{Z}$ (with maps $\mu_{mn}:\mathbb{Z}/p^{m}\mathbb{Z}\to\mathbb{Z}/p^{n}\mathbb{Z}$ given by 
$$(a+p^{m}\mathbb{Z})\mu_{mn}= a+p^{n}\mathbb{Z}\textup{, for each }n\leq m$$
 form a surjective inverse system.  
As $\phi^{i}_{j}$ and $\mu_{mn}$ commute (for each $i,j,m,n$) so do $\psi^{(\mathcal{U})}_{j}$ and $\mu_{mn}$.  
Now, we have surjective maps from $\prod_{i\in \mathbb{N}} (\mathbb{Z}/{p^i}\mathbb{Z})^{\alpha_i}$ to each term of the surjective inverse system, and all terms commute.  
By the universal property, these pull back to a unique surjective homomorphism, call it $\psi^{(\mathcal{U})}$, to the inverse limit.  
This map,
$$ \psi^{(\mathcal{U})}: \prod_{i\in \mathbb{N}} (\mathbb{Z}/{p^i}\mathbb{Z})^{\alpha_i} \to \mathbb{Z}_p$$
is given by
$$a\mapsto \lim_{j\in\mathbb{N}} ((a )\psi^{(\mathcal{U})}_{j})$$
(where we refer to the limit under the $p$-adic norm).  
Now $\mathbb{Z}_{p}$ is a torsion-free group, but $\prod_{i\in \mathbb{N}} (\mathbb{Z}/{p^i}\mathbb{Z})^{\alpha_i}$is dual-reduced, so $\psi^{(\mathcal{U})}$ cannot be continuous, and so has non-closed kernel.  
Consider the diagonal element $\eta=(x_{i})$, $x_{i}=\alpha_{i} \in \{0,1\}$.  
It is straightforward to see that 
$$\bar{\langle \eta\rangle}=\{(x_{i}) : x_m\mu_{m,n} = x_n, \forall m,n\in\mathbb{N}, \alpha_{m},\alpha_{n}\neq 0\}$$
and so $\bar{\langle \eta\rangle}\psi^{(\mathcal{U})}=\mathbb{Z}_p$.  
Now, $\ker\psi$ is a complement to $\bar{\langle\eta\rangle}$, so we have 
$$\prod_{i\in\mathbb{N}}(\mathbb{Z}/{p^i}\mathbb{Z})^{\alpha_i}=\text{ker}\psi^{(\mathcal{U})}\ltimes \bar{\langle\eta\rangle}= \text{ker}\psi^{(\mathcal{U})} \times \bar{\langle\eta\rangle}$$
as we are working with abelian groups. 
It follows that 
$$\ker\psi^{(\mathcal{U})}\cong \prod_{i\in \mathbb{N}} (\mathbb{Z}/{p^i}\mathbb{Z})^{\alpha_i} \slash \bar{\langle\eta\rangle} $$
 as abstract groups. 
But consider the endomorphism $\theta$ of $\prod_{i\in \mathbb{N}} (\mathbb{Z}/{p^i}\mathbb{Z})^{\alpha_i}$ defined by 
$$(x_{i})_{i\in\mathbb{N}}\theta= (x_{i}-x_{i+t_{i}}\mu_{i+t_i,i})_{i\in\mathbb{N}}$$
where $t_{i}$ is chosen to be minimal positive number such that $\alpha_{i+t_{i}} =1$ if $\alpha_{i}=1$, and to be $0$ when $\alpha_{i}=0$.  
It is surjective and has kernel precisely $\bar{\langle \eta\rangle}$, and so our quotient is topologically isomorphic to $\prod_{i\in \mathbb{N}} (\mathbb{Z}/{p^i}\mathbb{Z})^{\alpha_i}$.  
We have, for any sequence $(\alpha_i)$ in $\mathbb{N}\cup\{0,\aleph_0\}$ with infinitely many non-zero terms
$$\prod_{i\in\mathbb{N}}(\mathbb{Z}/{p^i}\mathbb{Z})^{\alpha_{i}}\cong \prod_{i\in\mathbb{N}}(\mathbb{Z}/{p^i}\mathbb{Z})^{\alpha_{i}} \times \mathbb{Z}_p$$
as abstract groups.  

We can write $\prod _{i\in\mathbb{N}} (\mathbb{Z}/{p^i}\mathbb{Z})^{\alpha_{i}}$ as 
$$\prod_{i\in I}\prod _{j\in\mathbb{N}} (\mathbb{Z}/{p^j}\mathbb{Z})^{\alpha^{(i)}_{j}}$$, for index set $I$ of arbitrary countable cardinality, such that each $(\alpha^{(i)}_{j})$ is a sequence in $\mathbb{N}\cup\{\aleph_{0}\}$ with infinitely many non-zero entries.  
Now, taking the Cartesian product of groups on the left and right hand sides of the above isomorphism completes the proof.  
\end{proof}

Note that $\psi^{(\mathcal{U})}$ above is in fact a ring homomorphism, as each homomorphism described above can be shown with little work to be a ring homomorphism.  
The direct product decomposition described above is not a decomposition as rings, as $\bar{\langle\eta\rangle}$ is not an ideal.

Note that a group can be written as $\prod _{i\in\mathbb{N}} C_{p^{\alpha_{i}}}$ for $(\alpha_{n})_{n\in\mathbb{N}}$ an unbounded sequence in $\mathbb{N}$ if and only if it can be written as $\prod _{i\in\mathbb{N}} (C_{p^i})^{\beta_{i}}$ for $(\beta_{i})$ some sequence in $\mathbb{N}\cup\{0,\aleph_{0}\}$ not tending to $0$.

Lemma \ref{NikHom} is not hard to strengthen.

\begin{Corollary}\label{NHI}
Let $(\alpha_{n})_{n\in\mathbb{N}}$ be an unbounded sequence in $\mathbb{N}$.  

Write $G= \prod _{i\in\mathbb{N}} C_{p^{\alpha_{i}}}$.  

If $y\in G\setminus (pG+t(G))$ then 
$$ G=K\times\bar{\langle y\rangle}$$ 
where $K\cong G$ (as abstract groups).
\end{Corollary}

\begin{proof}
Write $[x]_{i}$ for the sequence with $i$th entry $x$ and every other entry trivial.

Write $y=(y_{i})_{i\in \mathbb{N}}$.  
Each $y_{i}$ is either a generator for $C_{p^{\alpha_i}}$ or a non-generator.  
As $y \in G\setminus (pG+t(G))$ the set $N_{y}=\{i\in\mathbb{N} : \langle y_{i}\rangle = C_{p^{\alpha_i}}\}$ has no upper bound.  
This is because $z=(z_{i})$ is in $t(G)$ if and only if the set $\{o(z_{i}) : i\in\mathbb{N}\}$ is bounded and because $p$th powers in cyclic groups are not generators.

Suppose every $y_i$ is a generator.  
Then we can construct a continuous isomorphism $\theta :G \to  \prod _{i\in\mathbb{N}}\mathbb{Z}/{p^{\alpha_{i}}}\mathbb{Z}$ by sending each $[y_i]_{i}$ to $[1]_{i}$ and extending.  
Then Lemma \ref{NikHom} applied to $G\theta$ shows that $G=K\times\bar{\langle y\rangle}$, where $K$ is the pre-image in $G$ of the non-closed subgroup $\ker\psi^{(\mathcal{U})}$.  
(We use the notation given previously: $\psi^{(\mathcal{U})}$ is the non-continuous homomorphism defined in the proof of Lemma \ref{NikHom}.)   

Conversely, if $y$ has any non-generator entries, we consider the canonical projection $\pi_{N_{y}}: G \to \prod_{i\in N_{y}}C_{p^{\alpha_{i}}}$.  
Every entry of $y\pi_{N_{y}}$ is a generator, and we are in the situation described above.  

Hence $\prod_{i\in N_{y}}C_{p^{\alpha_{i}}} = K_{y}\times \bar{\langle y\pi_{N_{y}}\rangle}$, with $\prod_{i\in N_{y}}C_{p^{\alpha_{i}}}\cong K_{y}$, as abstract groups.  
We have 
$$K_{y}\pi_{N_{y}}^{-1}\cong  \prod_{i\in N_{y}}C_{p^{\alpha_{i}}}\times \prod_{i\in\mathbb{N}\setminus N_{y}} C_{p^{\alpha_{i}}}.$$  
Now, to show $K_{y}\pi_{N_{y}}^{-1}\times \bar{\langle y \rangle} =G$, it is sufficient to show that 
$$K_{y}\pi_{N_{y}}^{-1}\cap\bar{\langle y \rangle}=\{0\}.$$  

Suppose $\beta y\in K_{y}\pi_{N_{y}}^{-1}\cap\bar{\langle y \rangle}$.  
Then, for each $n\in N_{y}$, the $n$th entry of $\beta y$ must be trivial, and hence $\beta\equiv 0$ modulo $p^{\alpha_{n}}$.  
But $N_{y}$ is infinite and so $\beta=0$.
\end{proof}

We need another strengthening of this result to prove Theorem \ref{SAI}.

\begin{lem}\label{NHE}

Let $\alpha_{i}$ be an unbounded sequence in $\mathbb{N}$, and $H\cong\prod _{i\in\mathbb{N}} C_{p^{\alpha_{i}}}$.  
Let $A$ be an abstract abelian group containing $H$ with $t(A)\subseteq H$, and $A/H\cong C_{p^{k}}$, for some $k\in\mathbb{N}$.  
Then $A$ is isomorphic to $H$ as an abstract group.
\end{lem}

\begin{proof}
Choose any $x \in A$ such that $\langle x\rangle + H= A$: we have ${p^k}x=y \in H$.  
Note that $y\in H\setminus (t(A)+{p}H)$.  
Otherwise, $y={p}z+t$ for some torsion $t$ and $z\in H$.  
But then ${p^{k-1}}x-z$ is torsion, and so in $H$, as $p({p^{k-1}}x - z)=t$ is torsion.   
Hence $x$ has order at most $p^{k-1}$ in $A/H$ and so $H$ is of index at most $p^{k-1}$ in $A$, which is a contradiction.  

We can consider $H$ to be a pro-$p$ group under the product topology.  
Write $C$ to denote the subgroup of $H$ topologically generated by $y$, and $D$ to be the subgroup topologically generated by $x$.  
From the previous result, Corollary \ref{NHI}, we have that $H= K \times C$ where $K$ is isomorphic to $H$ as an abstract group.  
But, we can see that $D=\langle x, C \;\vert\; {p^{k}}x=y\rangle \cong C\cong \mathbb{Z}_{p}$ trivially.  
Now, we can see that 
$$A=\langle x \rangle + H= \langle H, x\rangle= \langle K\times C,  x\rangle= K\times D.$$  
Now, by Corollary \ref{NHI}, $A$ is isomorphic to $H$ as an abstract group.  
\end{proof}

\begin{proof}(of Theorem \ref{SAI})
By Theorem \ref{proUlm}, $G$ is topologically isomorphic to a group constructed as in the proof of Theorem \ref{Construction}.  
We now proceed by induction on $\tau$, the torsion type of $G$. 
The base case, $\tau=2$ holds as the closure of the torsion subgroup of an abelian pro-$p$ group must be Cartesian.   

If $\tau$ is a limit ordinal, then the proof of Theorem \ref{Construction} shows that $G$ is the product of groups of lower torsion types.  
By the inductive hypothesis, each of these is isomorphic to the closure of its torsion subgroup and so $G$ must be isomorphic to $\bar{t(G)}$.

If $\tau$ is a successor ordinal, by the argument of Proposition \ref{Prod}, we can write 
$$G=\prod_{i\in I} K_{i},$$
 where $I$ is not necessarily infinite, such that each $K_{i}$ has torsion type $\tau$, and each $(K_{i})_{T_{\tau-1}}$ cyclic.  
Hence, without loss of generality, we can assume that $G_{T_{\tau-1}}$ is a cyclic group, of order $p^{k}$ for some $k\in\mathbb{N}$.  
But then $T_{\tau-1}(G)$ is a subgroup of $G$, which, by definition contains $t(G)$, and  $G/T_{\tau-1}(G)\cong C_{p^{k}}$.  
But by the inductive hypothesis $T_{\tau-1}(G)$ is abstractly isomorphic to $\bar{t(G)}$, a Cartesian group of unbounded exponent.  
Now Lemma \ref{NHE} gives the result.  

This classifies all dual-reduced pro-$p$ groups and so, with Theorem \ref{Peeling} and Lemma \ref{NikHom}, we are done.\end{proof}
\end{section}

\begin{section}{Embeddings}\label{Universality}

The statement of Theorem \ref{SAI} leads one to ask if every countably-based abelian pro-$p$ group with unbounded torsion is isomorphic (as a profinite group) to a closed subgroup of the closure of its torsion subgroup.  
In fact, a much stronger result is true.
\begin{Theorem}
Let $A, B$ be countably-based abelian pro-$p$ groups.  
If $t(B)$ is of unbounded exponent, then $A$ is (topologically) isomorphic to a closed subgroup of $\bar{t(B)}$.   
\end{Theorem}
\begin{proof}
Every profinite group is isomorphic to a closed subgroup of the Cartesian product of its discrete images: so every countably-based abelian pro-$p$ group is a closed subgroup is a product of cyclic $p$-groups.  

For a collection $G_i$ of finite groups, we have  $\prod_i H_i \leq_{C} \prod_i G_i$ if $H_i \leq G_i$, for each $i$.  
So $\prod _{i\in\mathbb{N}} (C_{p^i})^{\alpha_{i}}\leq_{C} \prod _{i\in\mathbb{N}} (C_{p^i})^{\beta_{i}}$, for $(\alpha_{i}), (\beta_{i})$ unbounded sequences in $\mathbb{N}\cup\{\aleph_{0}\}$, such that $(\beta_{i})$ has with infinitely many non-zeros.  
Hence, by Theorem \ref{SAI} every countably-based abelian pro-$p$ group with unbounded torsion contains every countably-based abelian pro-$p$ group as a closed subgroup (of the closure of its torsion subgroup).  
\end{proof}
\end{section}

{\bf Acknowledgement.}

This work was done under an EPSRC Doctoral Training Award.  
The author would like to thank the referee and editor for their helpful comments, and to thank his Ph.D. supervisor, Dr. Nikolay Nikolov for many useful discussions on the material, without which this paper would not be possible.

\end{document}